\documentclass[12pt,a4paper]{article}
\usepackage{fancyhdr}
\usepackage[british]{babel}

\usepackage{amssymb, amsmath, amsthm, amsfonts, mathtools}
\usepackage[usenames,dvipsnames,svgnames,table]{xcolor}
\usepackage{graphicx,tikz,caption,subcaption}
\usetikzlibrary{patterns, shapes}
\usepackage{hyperref}
\usepackage{enumitem,verbatim}
\usepackage{multicol}
\usepackage{float}
\usepackage{cleveref}
\usepackage[margin=2.5cm]{geometry}
\usepackage{dsfont}
\usepackage{mathtools}
\usepackage{enumitem}
\usepackage{mathrsfs}
\definecolor{gray}{rgb}{0.25, 0.25, 0.25}

\newtheorem{theorem}{Theorem}[section]

\newtheorem{conjecture}[theorem]{Conjecture}

\newtheorem{fact}[theorem]{Fact}

\newtheorem*{theorem*}{Theorem}

\newcommand{\Tr}{\operatorname{Tr}}

\newcommand{\GL}{\mathrm{GL}}

\theoremstyle{definition}

\theoremstyle{plain}

\newtheorem{prop}[theorem]{Proposition}

\theoremstyle{definition}

\hypersetup{
    colorlinks=true,
    linkcolor=black,
    citecolor=black,
    urlcolor=blue
}
\title{Conjugacy Class Averages and Sidorenko's Conjecture}
\author{Yuqi Zhao\thanks{Email: \texttt{yuqi.zhao012@gmail.com}}}
\date{}

\begin{document}

\maketitle
\begin{abstract}
Sidorenko's conjecture asserts that for every bipartite graph $H$ and every graph $G$,
\[
    t(H,G)\geq t(K_2,G)^{e(H)}.
\]
A result of Szegedy shows that, in order to prove the conjecture, it suffices to verify the corresponding inequality on a special family of highly symmetric bipartite Cayley type hosts arising from symmetric groups. Motivated by this reduction, we study Cayley type bipartite kernels associated with functions on finite groups and their conjugacy class averages.

Our first result gives a reduction through conjugacy averaging: for a fixed bipartite graph $H$, if the $H$-density of each Cayley type host is at least the $H$-density of its conjugacy class average, then $H$ is strong Sidorenko, and hence Sidorenko. Our second result proves a Sidorenko-type inequality for 1-subdivision graphs on conjugacy-averaged Cayley kernels associated with arbitrary real-valued functions on finite groups.
\end{abstract}

\section{Introduction}

One of the central problems in extremal graph theory is to determine the minimum possible homomorphism density of a fixed graph $H$ in graphs with a prescribed edge density. A prominent conjecture, proposed independently by Sidorenko~\cite{sidorenko1993correlation,sidorenko1991inequalities} and Erd\H{o}s--Simonovits~\cite{erdos1984cube}, asserts that for every bipartite graph $H$, this minimum is attained, at the level of homomorphism densities, by the random graph with the same edge density.

Throughout the paper, all graphs are finite. For graphs $H$ and $G$, a \emph{homomorphism} from $H$ to $G$ is a map
\[
    \varphi:V(H)\to V(G)
\]
such that $\varphi(u)\varphi(v)\in E(G)$ whenever $uv\in E(H)$. We write $\hom(H,G)$ for the number of homomorphisms from $H$ to $G$. The \emph{homomorphism density} of $H$ in $G$ is
\[
    t(H,G)=\frac{\hom(H,G)}{|V(G)|^{|V(H)|}}.
\]
Sidorenko's conjecture can then be stated as follows.

\begin{conjecture}[Sidorenko's conjecture]\label{conj:sidorenko}
For every bipartite graph $H$ and every graph $G$,
\[
    t(H,G)\geq t(K_2,G)^{e(H)}.
\]
\end{conjecture}

A bipartite graph satisfying this inequality for all host graphs $G$ is called \emph{Sidorenko}. Sidorenko~\cite{sidorenko1993correlation} proved the conjecture for trees, even cycles, and complete bipartite graphs. Subsequent works have verified the conjecture for several further classes of bipartite graphs; for instance, Hatami~\cite{hatami2010graph} proved it for hypercubes, Conlon, Fox, and Sudakov~\cite{conlon2010approximate} proved it for bipartite graphs with one vertex complete to the other side, and Conlon, Kim, Lee, and Lee~\cite{conlon2018some} proved it for strongly tree-decomposable graphs.

We shall also use the bipartite, or asymmetric, version of homomorphism density. Let
\[
    H=(X\sqcup Y,E_H)
    \quad\text{and}\quad
    G=(U\sqcup V,E_G)
\]
be bipartite graphs with specified bipartitions. A \emph{bipartite homomorphism} from $H$ to $G$ is a homomorphism $\varphi:V(H)\to V(G)$ such that
\[
    \varphi(X)\subseteq U,
    \qquad
    \varphi(Y)\subseteq V.
\]
Let $\hom_{\mathrm{bip}}(H,G)$ be the number of bipartite homomorphisms from $H$ to $G$, and define
\[
    t_{\mathrm{bip}}(H,G)
    =
    \frac{\hom_{\mathrm{bip}}(H,G)}
    {|U|^{|X|}|V|^{|Y|}}.
\]
Thus the bipartite edge density of $G$ is
\[
    t_{\mathrm{bip}}(K_2,G)=\frac{|E_G|}{|U||V|}.
\]
We say that $H$ is \emph{strong Sidorenko} if
\[
    t_{\mathrm{bip}}(H,G)
    \geq
    t_{\mathrm{bip}}(K_2,G)^{e(H)}
\]
for every bipartite host graph $G$ with a specified bipartition. Strong Sidorenko implies Sidorenko. The asymmetric and directed formulations of Sidorenko-type inequalities are closely related; see, for example, Fox, Himwich, Mani, and Zhou~\cite{fox2022note}.

One motivation for our work is a reduction of Szegedy~\cite{szegedy2015sparse}, which shows that it suffices to prove Sidorenko-type inequalities on highly symmetric bipartite hosts. We shall use the following Cayley-form version of this reduction.

\begin{prop}[Szegedy's reduction, Cayley form]\label{prop:szegedy}
Let $H=(X\sqcup Y,E_H)$ be a bipartite graph. For $n\geq 1$, let $S_n$ be the symmetric group, and let $T_1,T_2\leq S_n$ be subgroups. Define the bipartite graph
\[
    G_{\mathrm{Cay}}(S_n;T_1,T_2)
    =
    (S_n^{(1)}\sqcup S_n^{(2)},E)
\]
as follows. The two vertex classes are two copies of $S_n$. Identifying each copy with $S_n$, for
$x\in S_n^{(1)}$ and $y\in S_n^{(2)}$ we put
\[
    xy\in E
    \quad\Longleftrightarrow\quad
    x^{-1}y\in T_1T_2,
\]
where
\[
    T_1T_2=\{t_1t_2:t_1\in T_1,\ t_2\in T_2\}.
\]
Equivalently, $G_{\mathrm{Cay}}(S_n;T_1,T_2)$ is the bipartite Cayley graph on two copies of $S_n$ with connection set $T_1T_2$.

Suppose that for every $n$ and every pair of subgroups $T_1,T_2\leq S_n$ we have
\[
    t_{\mathrm{bip}}\bigl(H,G_{\mathrm{Cay}}(S_n;T_1,T_2)\bigr)
    \geq
    t_{\mathrm{bip}}\bigl(K_2,G_{\mathrm{Cay}}(S_n;T_1,T_2)\bigr)^{e(H)}.
\]
Then $H$ is strong Sidorenko. In particular, $H$ is Sidorenko.
\end{prop}

This formulation is obtained from Szegedy's usual coset formulation by pulling the host graph back to two copies of $S_n$; this regular blow-up does not change normalized bipartite homomorphism densities.

We now introduce notation for Cayley type kernels associated with functions on finite groups. Let $\Gamma$ be a finite group and let
\[
    a:\Gamma\to\mathbb{R}
\]
be a real-valued function. We write
\[
    \operatorname{Cay}_{\mathrm{bip}}(\Gamma,a)
\]
for the bipartite Cayley kernel on two copies of $\Gamma$ whose value from $x$ on the left to $y$ on the right is
\[
    a(x^{-1}y).
\]
If $a$ is non-negative, this may be viewed as a weighted bipartite graph. In particular, if $A\subseteq\Gamma$ and $a=\mathbf{1}_A$, then this is the bipartite version of the directed Cayley graph with connection set $A$.

For a bipartite graph $H=(X\sqcup Y,E_H)$, define
\[
    t_{\mathrm{Cay}}(H;\Gamma,a)
    =
    \mathbb{E}_{\phi:X\to\Gamma,\ \psi:Y\to\Gamma}
    \prod_{xy\in E_H}
    a\bigl(\phi(x)^{-1}\psi(y)\bigr).
\]
When $a$ is the indicator function of a subset of $\Gamma$, this agrees with the usual bipartite homomorphism density into the corresponding bipartite Cayley graph:
\[
    t_{\mathrm{Cay}}(H;\Gamma,\mathbf{1}_A)
    =
    t_{\mathrm{bip}}\bigl(H,\operatorname{Cay}_{\mathrm{bip}}(\Gamma,\mathbf{1}_A)\bigr).
\]
In general, \(t_{\mathrm{Cay}}(H;\Gamma,a)\) should be regarded as the homomorphism density of \(H\) in the Cayley kernel associated with \(a\). In particular,
\[
    t_{\mathrm{Cay}}(K_2;\Gamma,a)
    =
    \mathbb{E}_{z\in\Gamma}a(z).
\]
With this notation,
\[
    G_{\mathrm{Cay}}(S_n;T_1,T_2)
    =
    \operatorname{Cay}_{\mathrm{bip}}\bigl(S_n,\mathbf{1}_{T_1T_2}\bigr).
\]
Thus Proposition~\ref{prop:szegedy} says that, in order to prove that $H$ is strong Sidorenko, it suffices to verify the Sidorenko inequality for the Cayley type bipartite hosts
\[
    \operatorname{Cay}_{\mathrm{bip}}\bigl(S_n,\mathbf{1}_{T_1T_2}\bigr),
    \qquad T_1,T_2\leq S_n.
\]

Throughout the paper, whenever a bipartite graph
\[
    H=(X\sqcup Y,E_H)
\]
is counted in a directed Cayley graph or Cayley kernel, we interpret it as being counted in the bipartite Cayley kernel
\[
    \operatorname{Cay}_{\mathrm{bip}}(\Gamma,a),
\]
with vertices of $X$ mapped to the left copy of $\Gamma$ and vertices of $Y$ mapped to the right copy of $\Gamma$.

We now define the conjugacy class average of a function on a finite group. For a function $a:\Gamma\to\mathbb{R}$, define
\[
    a^{\mathrm{cl}}(z)
    =
    \mathbb{E}_{\gamma\in\Gamma}
    a(\gamma^{-1}z\gamma).
\]
Thus $a^{\mathrm{cl}}$ is a class function, i.e. it is constant on conjugacy classes. We define the conjugacy class average of the Cayley kernel
\[
    \operatorname{Cay}_{\mathrm{bip}}(\Gamma,a)
\]
to be
\[
    \operatorname{Cay}_{\mathrm{bip}}^{\mathrm{cl}}(\Gamma,a)
    :=
    \operatorname{Cay}_{\mathrm{bip}}(\Gamma,a^{\mathrm{cl}}).
\]
Since conjugation preserves the uniform measure on $\Gamma$, the averaging operation preserves the edge density:
\[
    t_{\mathrm{Cay}}(K_2;\Gamma,a^{\mathrm{cl}})
    =
    t_{\mathrm{Cay}}(K_2;\Gamma,a)
    =
    \mathbb{E}_{z\in\Gamma}a(z).
\]
In particular, if $A\subseteq\Gamma$, then
\[
    t_{\mathrm{Cay}}\bigl(K_2;\Gamma,(\mathbf{1}_A)^{\mathrm{cl}}\bigr)
    =
    \frac{|A|}{|\Gamma|}.
\]

Our first main result says that Sidorenko's conjecture can be approached by comparing the Cayley type hosts appearing in Szegedy's reduction with their conjugacy class averages.

\begin{theorem}\label{thm:conj-average-reduction}
Let $H=(X\sqcup Y,E_H)$ be a bipartite graph. Suppose that for every finite group
$\Gamma$ and every non-negative function
\[
    a:\Gamma\to\mathbb{R}_{\geq 0},
\]
we have
\[
    t_{\mathrm{Cay}}(H;\Gamma,a)
    \geq
    t_{\mathrm{Cay}}\bigl(H;\Gamma,a^{\mathrm{cl}}\bigr).
\]
Then $H$ is strong Sidorenko. In particular, $H$ is Sidorenko.
\end{theorem}

By Proposition~\ref{prop:szegedy}, in order to prove Theorem~\ref{thm:conj-average-reduction} it is enough to apply the assumed comparison inequality to the Cayley type hosts appearing in Szegedy's reduction. Thus, in the proof, we only need the comparison for
\[
    \Gamma=S_n
    \qquad\text{and}\qquad
    a=\mathbf{1}_{T_1T_2},
\]
where $T_1,T_2\leq S_n$ are subgroups. In this case,
\[
    a^{\mathrm{cl}}=(\mathbf{1}_{T_1T_2})^{\mathrm{cl}}.
\]

We also prove a positive result for 1-subdivision graphs on conjugacy-averaged Cayley kernels. Recall that, for a graph $H_0$, the \emph{1-subdivision} $\operatorname{Sub}(H_0)$ is obtained from $H_0$ by replacing every edge $uv\in E(H_0)$ with a path
\[
    u-w_{uv}-v,
\]
where the new vertices $w_{uv}$ are distinct for different edges $uv\in E(H_0)$. Subdivision graphs play an important role in the study of Sidorenko's conjecture; see, for example, \cite{chen2024kohayakawa,conlon2018some,im2024sidorenko}.

\begin{theorem}[1-subdivisions on conjugacy-averaged Cayley kernels]\label{thm:subdivision-average}
Let $H_0$ be a finite graph, and let
\[
    H=\operatorname{Sub}(H_0)
\]
be its 1-subdivision. Then for every finite group $\Gamma$ and every real-valued function
\[
    a:\Gamma\to\mathbb{R},
\]
we have
\[
    t_{\mathrm{Cay}}\bigl(H;\Gamma,a^{\mathrm{cl}}\bigr)
    \geq
    t_{\mathrm{Cay}}\bigl(K_2;\Gamma,a^{\mathrm{cl}}\bigr)^{e(H)}.
\]
Equivalently,
\[
    t_{\mathrm{Cay}}\bigl(H;\Gamma,a^{\mathrm{cl}}\bigr)
    \geq
    \left(\mathbb{E}_{g\in\Gamma}a(g)\right)^{2e(H_0)}.
\]
\end{theorem}

Since every even subdivision is the 1-subdivision of another graph, Theorem~\ref{thm:subdivision-average} immediately applies to arbitrary even subdivisions. Indeed, if each edge of a graph $H_0$ is replaced by a path of length $2\ell_e$, then this graph is the 1-subdivision of the graph obtained from $H_0$ by replacing each edge $e$ by a path of length $\ell_e$.

Theorem~\ref{thm:subdivision-average} can be viewed as a generalization of \cite{zhao2025sidorenko}. In that paper, Sidorenko-type inequalities were proved for even subdivisions over finite abelian Cayley hosts; here the ambient group is allowed to be non-abelian, and the Cayley kernel is allowed to arise from an arbitrary real-valued class function.

\section{Preliminaries}

\subsection{Linear-algebraic tools}

All matrices in this paper are over the complex field $\mathbb{C}$. For a square matrix $A$, we write $A^*$ for its conjugate transpose. For matrices $A$ and $B$, we write $A\otimes B$ for their tensor product. We shall use the standard identity
\[
    \Tr(A\otimes B)=\Tr(A)\Tr(B),
\]
whenever $A$ and $B$ are square matrices.

\begin{fact}\label{fact:tensor-product-multiplication}
Let $k$ be a positive integer. For each $i\in[k]$, let $A_i$ and $B_i$ be square matrices of the same size. Then
\[
    \left(\bigotimes_{i=1}^k A_i\right)
    \left(\bigotimes_{i=1}^k B_i\right)
    =
    \bigotimes_{i=1}^k (A_iB_i).
\]
\end{fact}

A square matrix $A$ is called \emph{positive semidefinite} if $A=A^*$ and all eigenvalues of $A$ are non-negative. Equivalently,
\[
    \langle Av,v\rangle\geq 0
    \qquad\text{for every vector }v.
\]
A matrix $P$ is called an \emph{orthogonal projection} if
\[
    P^2=P
    \qquad\text{and}\qquad
    P=P^*.
\]
Every orthogonal projection is positive semidefinite.

We shall use the following elementary facts.

\begin{fact}\label{fact:psd-trace-and-projections}
Let $A$ and $B$ be positive semidefinite matrices of the same size. Then
\[
    \Tr(AB)\geq 0.
\]
Moreover, if $P$ and $Q$ are commuting orthogonal projections, then $PQ$ is an orthogonal projection.
\end{fact}

\begin{proof}
Since $A$ is positive semidefinite, it has a positive semidefinite square root $A^{1/2}$. Hence
\[
    \Tr(AB)=\Tr(A^{1/2}BA^{1/2}).
\]
The matrix $A^{1/2}BA^{1/2}$ is positive semidefinite, and therefore its trace is non-negative.

If $P$ and $Q$ are commuting orthogonal projections, then
\[
    (PQ)^2=PQPQ=P^2Q^2=PQ
\]
and
\[
    (PQ)^*=Q^*P^*=QP=PQ.
\]
Thus $PQ$ is an orthogonal projection.
\end{proof}

\subsection{Fourier analysis on finite groups}

Throughout this subsection, let $\Gamma$ be a finite group with identity element $e$. All representations are over $\mathbb{C}$. We write $\GL(V)$ for the group of invertible linear operators on a finite-dimensional complex vector space $V$.

A \emph{representation} of $\Gamma$ on $V$ is a homomorphism
\[
    \rho:\Gamma\to\GL(V).
\]
A representation $\rho$ is called \emph{unitary} if $V$ is equipped with a Hermitian inner product and each operator $\rho(g)$ is unitary. Equivalently, after choosing an orthonormal basis,
\[
    \rho(g)^{-1}=\rho(g)^*
    \qquad\text{for every }g\in\Gamma.
\]
A subspace $W\subseteq V$ is called $\Gamma$-\emph{invariant} if
\[
    \rho(g)W\subseteq W
    \qquad\text{for every }g\in\Gamma.
\]
A representation is \emph{irreducible} if its only invariant subspaces are $\{0\}$ and $V$.

Since $\Gamma$ is finite, every finite-dimensional representation of $\Gamma$ is unitarizable. Thus, throughout the paper, we choose all irreducible representations to be unitary. Let $\widehat{\Gamma}$ be a fixed set of representatives of the irreducible unitary representations of $\Gamma$, one from each equivalence class. For $\rho\in\widehat{\Gamma}$, write
\[
    d_\rho:=\dim\rho.
\]

Since $\Gamma$ is finite, we identify $L^2(\Gamma)$ with the space of all functions $f:\Gamma\to\mathbb{C}$, equipped with the normalized inner product
\[
    \langle f,h\rangle_{L^2(\Gamma)}
    :=
    \mathbb{E}_{x\in\Gamma}f(x)\overline{h(x)}
    =
    \frac{1}{|\Gamma|}\sum_{x\in\Gamma}f(x)\overline{h(x)}.
\]
Throughout the paper, we use the normalized averaging notation
\[
    \mathbb{E}_{x\in\Gamma}\varphi(x)
    :=
    \frac{1}{|\Gamma|}\sum_{x\in\Gamma}\varphi(x).
\]

For a unitary representation $\rho$ of $\Gamma$ and a function $f\in L^2(\Gamma)$, define the \emph{Fourier coefficient} of $f$ at $\rho$ by
\begin{equation}\label{eq:FT-def}
    \widehat f(\rho)
    :=
    \mathbb{E}_{x\in\Gamma}
    f(x)\rho(x)^{-1}
    =
    \mathbb{E}_{x\in\Gamma}
    f(x)\rho(x)^*.
\end{equation}
With this normalization, the Fourier inversion formula is
\begin{equation}\label{eq:Fourier-inversion}
    f(g)
    =
    \sum_{\rho\in\widehat{\Gamma}}
    d_\rho\,\Tr\!\bigl(\widehat f(\rho)\rho(g)\bigr)
    \qquad
    (g\in\Gamma).
\end{equation}

We first record the standard averaging projection associated with a subgroup.

\begin{fact}[Averaging over a subgroup]\label{fact:subgroup-projection}
Let $L\leq\Gamma$ be a subgroup, and let $\rho$ be a unitary representation of $\Gamma$. Define
\[
    P_L^\rho
    :=
    \frac{1}{|L|}\sum_{\ell\in L}\rho(\ell).
\]
Equivalently, if $\mathbf{1}_L$ denotes the indicator function of $L$, then
\[
    P_L^\rho
    =
    \frac{|\Gamma|}{|L|}\,\widehat{\mathbf{1}_L}(\rho)^*.
\]
Then $P_L^\rho$ is the orthogonal projection onto the $L$-fixed subspace
\[
    V_\rho^L
    :=
    \{v\in V_\rho:\rho(\ell)v=v\text{ for every }\ell\in L\}.
\]
\end{fact}

\begin{proof}
Since $L$ is a subgroup, the map $\ell\mapsto \ell^{-1}$ is a bijection of $L$. Using unitarity,
\[
    \sum_{\ell\in L}\rho(\ell)
    =
    \sum_{\ell\in L}\rho(\ell^{-1})
    =
    \sum_{\ell\in L}\rho(\ell)^*.
\]
Thus $P_L^\rho=(P_L^\rho)^*$. Moreover,
\[
    \left(\sum_{\ell\in L}\rho(\ell)\right)^2
    =
    \sum_{\ell_1,\ell_2\in L}\rho(\ell_1\ell_2)
    =
    |L|\sum_{\ell\in L}\rho(\ell),
\]
and hence $(P_L^\rho)^2=P_L^\rho$. Therefore $P_L^\rho$ is an orthogonal projection.

It remains to identify its image. If $v\in V_\rho^L$, then clearly $P_L^\rho v=v$. Conversely, if $P_L^\rho v=v$, then for any $h\in L$,
\[
    \rho(h)P_L^\rho v
    =
    \frac{1}{|L|}\sum_{\ell\in L}\rho(h\ell)v
    =
    \frac{1}{|L|}\sum_{\ell\in L}\rho(\ell)v
    =
    P_L^\rho v
    =
    v.
\]
Thus $v$ is fixed by every element of $L$.
\end{proof}

We next discuss class functions and conjugacy class averages.

A function $f:\Gamma\to\mathbb{C}$ is called a \emph{class function} if it is invariant under conjugation, i.e.
\[
    f(\gamma^{-1}x\gamma)=f(x)
    \qquad\text{for all }x,\gamma\in\Gamma.
\]
Equivalently, $f$ is constant on conjugacy classes.

For an arbitrary function $f:\Gamma\to\mathbb{C}$, define its \emph{conjugacy class average} by
\[
    f^{\mathrm{cl}}(x)
    :=
    \mathbb{E}_{\gamma\in\Gamma}
    f(\gamma^{-1}x\gamma).
\]
Then $f^{\mathrm{cl}}$ is a class function.

\begin{fact}[Class functions and conjugacy-class averaging]\label{fact:class-function-fourier}
A function $f:\Gamma\to\mathbb{C}$ is a class function if and only if, for every $\rho\in\widehat{\Gamma}$, the matrix $\widehat f(\rho)$ is a scalar multiple of the identity.

More precisely, if $f$ is a class function, then
\[
    \widehat f(\rho)=\lambda_\rho(f)I_{d_\rho},
\]
where
\[
    \lambda_\rho(f)
    =
    \frac{1}{d_\rho}\Tr\bigl(\widehat f(\rho)\bigr)
    =
    \frac{1}{d_\rho}
    \mathbb{E}_{x\in\Gamma}
    f(x)\chi_\rho(x^{-1}),
\]
and $\chi_\rho(x):=\Tr(\rho(x))$ is the character of $\rho$.

Moreover, for every $f:\Gamma\to\mathbb{C}$ and every $\rho\in\widehat{\Gamma}$,
\begin{equation}\label{eq:class-average-fourier}
    \widehat{f^{\mathrm{cl}}}(\rho)
    =
    \frac{1}{d_\rho}\Tr\bigl(\widehat f(\rho)\bigr)I_{d_\rho}.
\end{equation}
\end{fact}

\begin{proof}
Suppose first that $f$ is a class function. For any $\gamma\in\Gamma$,
\[
    \rho(\gamma)\widehat f(\rho)\rho(\gamma)^{-1}
    =
    \mathbb{E}_{x\in\Gamma}
    f(x)\rho(\gamma x^{-1}\gamma^{-1}).
\]
After the change of variables $y=\gamma x\gamma^{-1}$, and using conjugation-invariance of $f$, the right-hand side becomes
\[
    \mathbb{E}_{y\in\Gamma}f(y)\rho(y^{-1})
    =
    \widehat f(\rho).
\]
Thus $\widehat f(\rho)$ commutes with $\rho(\gamma)$ for every $\gamma\in\Gamma$. By Schur's lemma, $\widehat f(\rho)$ is a scalar multiple of the identity.

Conversely, suppose that $\widehat f(\rho)$ is a scalar multiple of the identity for every $\rho\in\widehat{\Gamma}$. By Fourier inversion,
\[
    f(g)
    =
    \sum_{\rho\in\widehat{\Gamma}}
    d_\rho\,\lambda_\rho(f)\chi_\rho(g).
\]
Since every character is a class function, $f$ is a class function.

It remains to prove \eqref{eq:class-average-fourier}. The function $f^{\mathrm{cl}}$ is a class function, so $\widehat{f^{\mathrm{cl}}}(\rho)$ is a scalar matrix. Taking traces and using the definition of $f^{\mathrm{cl}}$,
\[
\begin{aligned}
    \Tr\bigl(\widehat{f^{\mathrm{cl}}}(\rho)\bigr)
    &=
    \mathbb{E}_{x\in\Gamma}
    f^{\mathrm{cl}}(x)\chi_\rho(x^{-1})\\
    &=
    \mathbb{E}_{x,\gamma\in\Gamma}
    f(\gamma^{-1}x\gamma)\chi_\rho(x^{-1})\\
    &=
    \mathbb{E}_{y,\gamma\in\Gamma}
    f(y)\chi_\rho(\gamma y^{-1}\gamma^{-1})\\
    &=
    \mathbb{E}_{y\in\Gamma}
    f(y)\chi_\rho(y^{-1})
    =
    \Tr\bigl(\widehat f(\rho)\bigr),
\end{aligned}
\]
where we used that characters are class functions. This proves \eqref{eq:class-average-fourier}.
\end{proof}

We shall also use the following standard consequence of Schur's lemma; see, for instance, \cite{serre1977linear}.

\begin{fact}[Tensor averages of irreducible representations]\label{fact:tensor-average-dual}
Let $\rho_1,\rho_2\in\widehat{\Gamma}$ be irreducible unitary representations. Then
\[
    \mathbb{E}_{g\in\Gamma}\rho_1(g)\otimes\rho_2(g)
\]
is the orthogonal projection onto the invariant subspace
\[
    (V_{\rho_1}\otimes V_{\rho_2})^\Gamma.
\]
In particular,
\[
    \mathbb{E}_{g\in\Gamma}\rho_1(g)\otimes\rho_2(g)\neq 0
\]
if and only if
\[
    \rho_2\simeq\rho_1^\vee,
\]
where $\rho_1^\vee$ denotes the contragredient representation of $\rho_1$.
\end{fact}

For real-valued class functions, the Fourier scalars at dual representations are complex conjugates.

\begin{fact}[Dual representations and real-valued class functions]\label{fact:dual-class-coefficients}
Let $f:\Gamma\to\mathbb{R}$ be a real-valued class function. For every irreducible unitary representation $\rho$, write
\[
    \widehat f(\rho)=\lambda_\rho(f)I_{d_\rho}.
\]
Then
\[
    \lambda_{\rho^\vee}(f)
    =
    \overline{\lambda_\rho(f)}.
\]
Here, if the contragredient representation $\rho^\vee$ is not the chosen representative in $\widehat{\Gamma}$, the notation $\lambda_{\rho^\vee}(f)$ means the scalar attached to the unique representative of its equivalence class.
\end{fact}

\begin{proof}
By Fact~\ref{fact:class-function-fourier},
\[
    \lambda_\rho(f)
    =
    \frac{1}{d_\rho}
    \mathbb{E}_{x\in\Gamma}
    f(x)\chi_\rho(x^{-1}).
\]
Since
\[
    \chi_{\rho^\vee}(x)=\chi_\rho(x^{-1}),
\]
we have
\[
    \lambda_{\rho^\vee}(f)
    =
    \frac{1}{d_\rho}
    \mathbb{E}_{x\in\Gamma}
    f(x)\chi_\rho(x).
\]
Because $f$ is real-valued and $\rho$ is unitary,
\[
    \overline{\chi_\rho(x^{-1})}
    =
    \chi_\rho(x).
\]
Therefore
\[
    \lambda_{\rho^\vee}(f)
    =
    \overline{\lambda_\rho(f)}.
\]
\end{proof}

The following consequence for products of two subgroups will be used in the proof of the reduction theorem.

\begin{fact}[Products of two subgroups]\label{fact:product-subgroups-class-average}
Let $L,M\leq\Gamma$ be subgroups, and let
\[
    F:=\mathbf{1}_{LM}
\]
be the indicator function of the product set
\[
    LM=\{\ell m:\ell\in L,\ m\in M\}.
\]
Then, for every $\rho\in\widehat{\Gamma}$,
\[
    \widehat{F^{\mathrm{cl}}}(\rho)
    =
    c_\rho(L,M)I_{d_\rho},
\]
where
\[
    c_\rho(L,M)
    =
    \frac{|LM|}{|\Gamma|\,d_\rho}
    \Tr\!\bigl(P_M^\rho P_L^\rho\bigr)
    =
    \frac{|LM|}{|\Gamma|\,d_\rho}
    \Tr\!\bigl(P_L^\rho P_M^\rho\bigr)
    \geq 0.
\]
In particular, the Fourier coefficients of $F^{\mathrm{cl}}$ are non-negative scalar multiples of the identity.
\end{fact}

\begin{proof}
Every element of $LM$ has exactly $|L\cap M|$ representations as $\ell m$ with $\ell\in L$ and $m\in M$. Hence
\[
    |LM|=\frac{|L||M|}{|L\cap M|}.
\]
Therefore
\[
\begin{aligned}
    \sum_{x\in LM}\rho(x)
    &=
    \frac{1}{|L\cap M|}
    \sum_{\ell\in L}\sum_{m\in M}\rho(\ell m)\\
    &=
    \frac{|L||M|}{|L\cap M|}
    P_L^\rho P_M^\rho
    =
    |LM|\,P_L^\rho P_M^\rho.
\end{aligned}
\]
Since $F$ is real-valued,
\[
\begin{aligned}
    \widehat F(\rho)^*
    &=
    \mathbb{E}_{x\in\Gamma}F(x)\rho(x)\\
    &=
    \frac{1}{|\Gamma|}
    \sum_{x\in LM}\rho(x)
    =
    \frac{|LM|}{|\Gamma|}P_L^\rho P_M^\rho.
\end{aligned}
\]
Thus
\[
    \widehat F(\rho)
    =
    \frac{|LM|}{|\Gamma|}P_M^\rho P_L^\rho.
\]
Taking traces gives
\[
    \Tr\bigl(\widehat F(\rho)\bigr)
    =
    \frac{|LM|}{|\Gamma|}
    \Tr\!\bigl(P_M^\rho P_L^\rho\bigr).
\]
Using \eqref{eq:class-average-fourier}, we obtain
\[
    \widehat{F^{\mathrm{cl}}}(\rho)
    =
    \frac{1}{d_\rho}
    \Tr\bigl(\widehat F(\rho)\bigr)I_{d_\rho}
    =
    \frac{|LM|}{|\Gamma|\,d_\rho}
    \Tr\!\bigl(P_M^\rho P_L^\rho\bigr)I_{d_\rho}.
\]
Finally, $P_L^\rho$ and $P_M^\rho$ are orthogonal projections by Fact~\ref{fact:subgroup-projection}; hence they are positive semidefinite. By Fact~\ref{fact:psd-trace-and-projections},
\[
    \Tr(P_M^\rho P_L^\rho)\geq 0.
\]
This proves the claim.
\end{proof}

If, in addition, $\mathbf{1}_{LM}$ itself is a class function, then Fact~\ref{fact:class-function-fourier} and Fact~\ref{fact:product-subgroups-class-average} imply that
\[
    \widehat{\mathbf{1}_{LM}}(\rho)
    =
    c_\rho(L,M)I_{d_\rho}
    \qquad\text{with }c_\rho(L,M)\geq 0
\]
for every $\rho\in\widehat{\Gamma}$.

\section{Conjugacy class averaging}

In this section we prove the reduction through conjugacy averaging. The main point is that, after taking conjugacy class averages of product sets of two subgroups, the Fourier coefficients become non-negative scalar matrices. This allows us to prove the Sidorenko inequality term by term after Fourier expansion.

We first prove a positivity statement for conjugacy-averaged product sets of two subgroups.

\begin{prop}\label{prop:averaged-product-host-sidorenko}
Let $\Gamma$ be a finite group, let $L,M\leq \Gamma$ be subgroups, and set
\[
    a=(\mathbf{1}_{LM})^{\mathrm{cl}}.
\]
Then for every bipartite graph $H=(X\sqcup Y,E_H)$,
\[
    t_{\mathrm{Cay}}(H;\Gamma,a)
    \geq
    t_{\mathrm{Cay}}(K_2;\Gamma,a)^{e(H)}.
\]
Equivalently,
\[
    t_{\mathrm{Cay}}(H;\Gamma,(\mathbf{1}_{LM})^{\mathrm{cl}})
    \geq
    \left(\frac{|LM|}{|\Gamma|}\right)^{e(H)}.
\]
\end{prop}

\begin{proof}
Let
\[
    E_H=\{e_1,\dots,e_m\},
    \qquad m=e(H),
\]
and write each edge as
\[
    e_r=x_r y_r,
    \qquad x_r\in X,\quad y_r\in Y.
\]
For each $x\in X$ and $y\in Y$, define the incidence matrices
\[
    M_X(x,r)=
    \begin{cases}
    1,&\text{if }x=x_r,\\
    0,&\text{otherwise,}
    \end{cases}
\]
and
\[
    M_Y(y,r)=
    \begin{cases}
    1,&\text{if }y=y_r,\\
    0,&\text{otherwise.}
    \end{cases}
\]
Thus each column $r$ has exactly one $1$ in $M_X$ and exactly one $1$ in $M_Y$.

By Fact~\ref{fact:product-subgroups-class-average}, for every irreducible representation $\rho\in\widehat{\Gamma}$,
\[
    \widehat a(\rho)=c_\rho I_{d_\rho},
    \qquad c_\rho\geq 0.
\]
By the Fourier inversion formula \eqref{eq:Fourier-inversion}, for every $g\in\Gamma$ we have
\[
    a(g)
    =
    \sum_{\rho\in\widehat{\Gamma}}
    d_\rho\,c_\rho\,\Tr(\rho(g)).
\]
Therefore
\begin{align*}
t_{\mathrm{Cay}}(H;\Gamma,a)
&=
\mathbb{E}_{\phi:X\to\Gamma,\ \psi:Y\to\Gamma}
\prod_{r=1}^m
a\bigl(\phi(x_r)^{-1}\psi(y_r)\bigr)\\
&=
\sum_{\rho_1,\dots,\rho_m\in\widehat{\Gamma}}
\left(\prod_{r=1}^m d_{\rho_r}c_{\rho_r}\right)
\mathcal{T}(\rho_1,\dots,\rho_m),
\end{align*}
where
\[
    \mathcal{T}(\rho_1,\dots,\rho_m)
    :=
    \mathbb{E}_{\phi,\psi}
    \prod_{r=1}^m
    \Tr\!\left(
        \rho_r\bigl(\phi(x_r)^{-1}\psi(y_r)\bigr)
    \right).
\]
We shall prove that
\[
    \mathcal{T}(\rho_1,\dots,\rho_m)\geq 0
\]
for every choice of $\rho_1,\dots,\rho_m$. Since each coefficient
\[
    \prod_{r=1}^m d_{\rho_r}c_{\rho_r}
\]
is non-negative, this will imply that every Fourier contribution is non-negative.

Fix $\rho_1,\dots,\rho_m\in\widehat{\Gamma}$, and let
\[
    V:=\bigotimes_{r=1}^m V_{\rho_r}.
\]
For fixed maps $\phi:X\to\Gamma$ and $\psi:Y\to\Gamma$, using the identities
\[
    \rho_r\bigl(\phi(x_r)^{-1}\psi(y_r)\bigr)
    =
    \rho_r(\phi(x_r)^{-1})\rho_r(\psi(y_r))
\]
and Fact~\ref{fact:tensor-product-multiplication}, we get
\begin{align*}
&\prod_{r=1}^m
    \Tr\!\left(
        \rho_r\bigl(\phi(x_r)^{-1}\psi(y_r)\bigr)
    \right)\\
&\qquad =
\Tr\!\left[
    \left(
        \bigotimes_{r=1}^m
        \rho_r(\phi(x_r)^{-1})
    \right)
    \left(
        \bigotimes_{r=1}^m
        \rho_r(\psi(y_r))
    \right)
\right].
\end{align*}

For each $x\in X$, define
\[
    P_x
    :=
    \mathbb{E}_{g\in\Gamma}
    \bigotimes_{r=1}^m
    \rho_r(g^{-M_X(x,r)}).
\]
Here the convention is that $\rho_r(g^0)=I_{d_{\rho_r}}$ and
$\rho_r(g^{-1})=\rho_r(g)^{-1}$.

Similarly, for each $y\in Y$, define
\[
    Q_y
    :=
    \mathbb{E}_{g\in\Gamma}
    \bigotimes_{r=1}^m
    \rho_r(g^{M_Y(y,r)}).
\]

Each $P_x$ is an orthogonal projection. Indeed, after the change of variables $g\mapsto g^{-1}$, $P_x$ is the average over $\Gamma$ of the unitary representation
\[
    g\mapsto
    \bigotimes_{r=1}^m
    \rho_r(g^{M_X(x,r)}).
\]
Hence Fact~\ref{fact:subgroup-projection}, applied with the subgroup $\Gamma\leq\Gamma$, shows that $P_x$ is an orthogonal projection. Similarly, each $Q_y$ is an orthogonal projection.

Moreover, the projections $\{P_x:x\in X\}$ commute with one another. Indeed, if $x,x'\in X$ are distinct, then for every edge index $r$,
\[
    M_X(x,r)M_X(x',r)=0,
\]
because an edge has only one endpoint in $X$. Hence, for every $g,h\in\Gamma$,
\[
\left(
    \bigotimes_{r=1}^m \rho_r(g^{-M_X(x,r)})
\right)
\left(
    \bigotimes_{r=1}^m \rho_r(h^{-M_X(x',r)})
\right)
=
\left(
    \bigotimes_{r=1}^m \rho_r(h^{-M_X(x',r)})
\right)
\left(
    \bigotimes_{r=1}^m \rho_r(g^{-M_X(x,r)})
\right),
\]
since in each tensor factor at least one of the two matrices is the identity. Averaging over $g$ and $h$ gives
\[
    P_xP_{x'}=P_{x'}P_x.
\]
Therefore
\[
    P_X:=\prod_{x\in X}P_x
\]
is an orthogonal projection by Fact~\ref{fact:psd-trace-and-projections}.

The same argument shows that the projections $\{Q_y:y\in Y\}$ commute with one another, and hence
\[
    Q_Y:=\prod_{y\in Y}Q_y
\]
is an orthogonal projection.

Now we compute $\mathcal{T}(\rho_1,\dots,\rho_m)$. By linearity of trace and independence of the variables $\phi(x)$ and $\psi(y)$,
\begin{align*}
\mathcal{T}(\rho_1,\dots,\rho_m)
&=
\Tr\!\left[
    \mathbb{E}_{\phi}
    \bigotimes_{r=1}^m
    \rho_r(\phi(x_r)^{-1})
    \cdot
    \mathbb{E}_{\psi}
    \bigotimes_{r=1}^m
    \rho_r(\psi(y_r))
\right]\\
&=
\Tr(P_XQ_Y).
\end{align*}
Since $P_X$ and $Q_Y$ are orthogonal projections, they are positive semidefinite. Hence, by Fact~\ref{fact:psd-trace-and-projections},
\[
    \Tr(P_XQ_Y)\geq 0.
\]
Thus
\[
    \mathcal{T}(\rho_1,\dots,\rho_m)\geq 0
\]
for every $\rho_1,\dots,\rho_m\in\widehat{\Gamma}$.

It remains to identify the contribution coming from the trivial representation. Let
\(\rho_{\mathrm{triv}}\) denote the trivial representation of \(\Gamma\). Then
\(d_{\rho_{\mathrm{triv}}}=1\), and \(\rho_{\mathrm{triv}}(g)=1\) for every
\(g\in\Gamma\). Hence
\[
    c_{\rho_{\mathrm{triv}}}
    =
    \widehat a(\rho_{\mathrm{triv}})
    =
    \mathbb{E}_{g\in\Gamma}a(g).
\]
Since \(a=(\mathbf{1}_{LM})^{\mathrm{cl}}\), conjugacy averaging preserves the
uniform average, and therefore
\[
    \mathbb{E}_{g\in\Gamma}a(g)
    =
    \mathbb{E}_{g\in\Gamma}(\mathbf{1}_{LM})^{\mathrm{cl}}(g)
    =
    \mathbb{E}_{g\in\Gamma}\mathbf{1}_{LM}(g)
    =
    \frac{|LM|}{|\Gamma|}.
\]
Thus
\[
    c_{\rho_{\mathrm{triv}}}
    =
    \frac{|LM|}{|\Gamma|}.
\]

Now consider the summand in the Fourier expansion for which
\[
    \rho_1=\rho_2=\cdots=\rho_m=\rho_{\mathrm{triv}}.
\]
For this choice,
\[
    \prod_{r=1}^m d_{\rho_r}c_{\rho_r}
    =
    \left(\frac{|LM|}{|\Gamma|}\right)^m.
\]
Moreover,
\[
\begin{aligned}
    \mathcal{T}(\rho_{\mathrm{triv}},\dots,\rho_{\mathrm{triv}})
    &=
    \mathbb{E}_{\phi:X\to\Gamma,\ \psi:Y\to\Gamma}
    \prod_{r=1}^m
    \Tr\!\left(
        \rho_{\mathrm{triv}}\bigl(\phi(x_r)^{-1}\psi(y_r)\bigr)
    \right)  \\
    &=
    \mathbb{E}_{\phi:X\to\Gamma,\ \psi:Y\to\Gamma}
    \prod_{r=1}^m 1
    =
    1.
\end{aligned}
\]
Therefore the trivial-representation contribution to
\(t_{\mathrm{Cay}}(H;\Gamma,a)\) is exactly
\[
    \left(\frac{|LM|}{|\Gamma|}\right)^m.
\]

Since all other Fourier contributions have been shown to be non-negative, we obtain
\[
    t_{\mathrm{Cay}}(H;\Gamma,a)
    \geq
    \left(\frac{|LM|}{|\Gamma|}\right)^m.
\]
Finally,
\[
    t_{\mathrm{Cay}}(K_2;\Gamma,a)
    =
    \mathbb{E}_{g\in\Gamma}a(g)
    =
    \frac{|LM|}{|\Gamma|}.
\]
As \(m=e(H)\), this gives
\[
    t_{\mathrm{Cay}}(H;\Gamma,a)
    \geq
    t_{\mathrm{Cay}}(K_2;\Gamma,a)^{e(H)}.
\]
This proves the proposition.
\end{proof}

We can now prove theorem~\ref{thm:conj-average-reduction}.

\begin{proof}[Proof of Theorem~\ref{thm:conj-average-reduction}]
Let $H=(X\sqcup Y,E_H)$ be a bipartite graph, and suppose that for every $n\geq 1$ and every pair of subgroups $T_1,T_2\leq S_n$,
\[
    t_{\mathrm{Cay}}\bigl(H;S_n,\mathbf{1}_{T_1T_2}\bigr)
    \geq
    t_{\mathrm{Cay}}\bigl(H;S_n,(\mathbf{1}_{T_1T_2})^{\mathrm{cl}}\bigr).
\]
We shall prove that $H$ is strong Sidorenko.

Fix $n$ and subgroups $T_1,T_2\leq S_n$. Applying Proposition~\ref{prop:averaged-product-host-sidorenko} with
\[
    \Gamma=S_n,\qquad L=T_1,\qquad M=T_2,
\]
we get
\[
    t_{\mathrm{Cay}}\bigl(H;S_n,(\mathbf{1}_{T_1T_2})^{\mathrm{cl}}\bigr)
    \geq
    \left(\frac{|T_1T_2|}{|S_n|}\right)^{e(H)}.
\]
By the assumed comparison inequality,
\[
    t_{\mathrm{Cay}}\bigl(H;S_n,\mathbf{1}_{T_1T_2}\bigr)
    \geq
    t_{\mathrm{Cay}}\bigl(H;S_n,(\mathbf{1}_{T_1T_2})^{\mathrm{cl}}\bigr).
\]
Therefore
\[
    t_{\mathrm{Cay}}\bigl(H;S_n,\mathbf{1}_{T_1T_2}\bigr)
    \geq
    \left(\frac{|T_1T_2|}{|S_n|}\right)^{e(H)}.
\]
But
\[
    t_{\mathrm{Cay}}\bigl(K_2;S_n,\mathbf{1}_{T_1T_2}\bigr)
    =
    \frac{|T_1T_2|}{|S_n|}.
\]
Hence
\[
    t_{\mathrm{Cay}}\bigl(H;S_n,\mathbf{1}_{T_1T_2}\bigr)
    \geq
    t_{\mathrm{Cay}}\bigl(K_2;S_n,\mathbf{1}_{T_1T_2}\bigr)^{e(H)}
\]
for every $n$ and every pair of subgroups $T_1,T_2\leq S_n$.

By Proposition~\ref{prop:szegedy}, this implies that $H$ is strong Sidorenko. In particular, $H$ is Sidorenko.
\end{proof}
\section{1-subdivisions on conjugacy-averaged Cayley kernels}

In this section we prove Theorem~\ref{thm:subdivision-average}. The proof is again based on Fourier expansion. The additional feature for 1-subdivisions is that each subdivision vertex has degree two, and the average over this vertex forces the two corresponding irreducible representations to be contragredient. Since the kernel is a real-valued class function, the two Fourier scalars are complex conjugates of each other.

\begin{proof}[Proof of Theorem~\ref{thm:subdivision-average}]
Let
\[
    H_0=(V,E)
\]
be a finite graph, and write
\[
    m:=|E|.
\]
Let
\[
    H=\operatorname{Sub}(H_0)
\]
be its 1-subdivision. We regard \(H\) as a bipartite graph with bipartition
\[
    X=V,
    \qquad
    Y=E,
\]
where each edge \(e\in E\) is viewed as the subdivision vertex corresponding to \(e\). If \(e=uv\in E\), then in \(H\) the subdivision vertex \(e\in Y\) is adjacent to the two vertices \(u,v\in X\).

Set
\[
    b:=a^{\mathrm{cl}}.
\]
Then \(b\) is a real-valued class function on \(\Gamma\). By Fact~\ref{fact:class-function-fourier}, for every
\(\rho\in\widehat{\Gamma}\) we may write
\[
    \widehat b(\rho)=\lambda_\rho I_{d_\rho}.
\]
Moreover, since \(b\) is real-valued, Fact~\ref{fact:dual-class-coefficients} gives
\[
    \lambda_{\rho^\vee}=\overline{\lambda_\rho}.
\]

By Fourier inversion \eqref{eq:Fourier-inversion}, for every \(g\in\Gamma\),
\[
    b(g)
    =
    \sum_{\rho\in\widehat{\Gamma}}
    d_\rho \lambda_\rho \Tr(\rho(g)).
\]

Let
\[
    I(H_0):=\{(v,e): v\in V,\ e\in E,\ v\in e\}
\]
be the incidence set of \(H_0\). Since \(H\) is the 1-subdivision of \(H_0\), its edge set is naturally identified with \(I(H_0)\). Thus
\[
    |I(H_0)|=2m=e(H).
\]
For maps
\[
    \phi:V\to\Gamma,
    \qquad
    \psi:E\to\Gamma,
\]
we have
\[
    t_{\mathrm{Cay}}(H;\Gamma,b)
    =
    \mathbb{E}_{\phi,\psi}
    \prod_{(v,e)\in I(H_0)}
    b\bigl(\phi(v)^{-1}\psi(e)\bigr).
\]
Expanding each factor by Fourier inversion gives
\[
    t_{\mathrm{Cay}}(H;\Gamma,b)
    =
    \sum_{\boldsymbol{\rho}\in\widehat{\Gamma}^{I(H_0)}}
    C(\boldsymbol{\rho})\,
    \mathcal{T}(\boldsymbol{\rho}),
\]
where
\[
    C(\boldsymbol{\rho})
    :=
    \prod_{(v,e)\in I(H_0)}
    d_{\rho_{v,e}}\lambda_{\rho_{v,e}},
\]
and
\[
    \mathcal{T}(\boldsymbol{\rho})
    :=
    \mathbb{E}_{\phi,\psi}
    \prod_{(v,e)\in I(H_0)}
    \Tr\!\left(
        \rho_{v,e}\bigl(\phi(v)^{-1}\psi(e)\bigr)
    \right).
\]

We now analyze \(\mathcal{T}(\boldsymbol{\rho})\). For a fixed choice
\[
    \boldsymbol{\rho}=(\rho_{v,e})_{(v,e)\in I(H_0)}
\]
of irreducible representations, set
\[
    W_{\boldsymbol{\rho}}
    :=
    \bigotimes_{(v,e)\in I(H_0)} V_{\rho_{v,e}}.
\]
Using Fact~\ref{fact:tensor-product-multiplication} and the identity
\[
    \rho_{v,e}\bigl(\phi(v)^{-1}\psi(e)\bigr)
    =
    \rho_{v,e}(\phi(v)^{-1})\rho_{v,e}(\psi(e)),
\]
we can rewrite the product of traces as a single trace:
\[
\begin{aligned}
&\prod_{(v,e)\in I(H_0)}
    \Tr\!\left(
        \rho_{v,e}\bigl(\phi(v)^{-1}\psi(e)\bigr)
    \right)\\
&\qquad =
\Tr\!\left[
    \left(
        \bigotimes_{(v,e)\in I(H_0)}
        \rho_{v,e}(\phi(v)^{-1})
    \right)
    \left(
        \bigotimes_{(v,e)\in I(H_0)}
        \rho_{v,e}(\psi(e))
    \right)
\right].
\end{aligned}
\]

For each \(v\in V\), define an operator \(P_v\) on \(W_{\boldsymbol{\rho}}\) by
\[
    P_v
    :=
    \mathbb{E}_{g\in\Gamma}
    \bigotimes_{(w,e)\in I(H_0)}
    R_{w,e}^{v}(g),
\]
where
\[
    R_{w,e}^{v}(g)
    :=
    \begin{cases}
    \rho_{w,e}(g^{-1}), & \text{if } w=v,\\
    I_{d_{\rho_{w,e}}}, & \text{if } w\neq v.
    \end{cases}
\]
Similarly, for each \(e\in E\), define an operator \(Q_e\) on \(W_{\boldsymbol{\rho}}\) by
\[
    Q_e
    :=
    \mathbb{E}_{g\in\Gamma}
    \bigotimes_{(v,f)\in I(H_0)}
    S_{v,f}^{e}(g),
\]
where
\[
    S_{v,f}^{e}(g)
    :=
    \begin{cases}
    \rho_{v,f}(g), & \text{if } f=e,\\
    I_{d_{\rho_{v,f}}}, & \text{if } f\neq e.
    \end{cases}
\]

Each \(P_v\) is the average over a unitary representation of \(\Gamma\), hence it is an orthogonal projection by Fact~\ref{fact:subgroup-projection}, applied to the subgroup \(\Gamma\leq\Gamma\). Similarly, each \(Q_e\) is an orthogonal projection.

Moreover, the projections \(P_v\) commute with one another. Indeed, if \(v\neq v'\), then for each incidence \((w,e)\) at most one of \(w=v\) and \(w=v'\) can hold. Hence in every tensor factor at least one of the two corresponding matrices is the identity. It follows that
\[
    P_vP_{v'}=P_{v'}P_v.
\]
Therefore
\[
    P_V:=\prod_{v\in V}P_v
\]
is an orthogonal projection by Fact~\ref{fact:psd-trace-and-projections}. Similarly, the projections \(Q_e\) commute with one another, since distinct edges \(e\neq e'\) correspond to disjoint sets of incidence tensor factors. Hence
\[
    Q_E:=\prod_{e\in E}Q_e
\]
is also an orthogonal projection.

By independence of the variables \(\phi(v)\) and \(\psi(e)\), and by linearity of trace, we obtain
\[
    \mathcal{T}(\boldsymbol{\rho})
    =
    \Tr(P_VQ_E).
\]
Since \(P_V\) and \(Q_E\) are orthogonal projections, they are positive semidefinite. Therefore, by Fact~\ref{fact:psd-trace-and-projections},
\[
    \mathcal{T}(\boldsymbol{\rho})=\Tr(P_VQ_E)\geq 0.
\]

It remains to identify when the coefficient \(C(\boldsymbol{\rho})\) contributes. For an edge \(e=uv\in E\), the operator \(Q_e\) acts nontrivially only on the two tensor factors corresponding to \((u,e)\) and \((v,e)\). On these two factors it is
\[
    \mathbb{E}_{g\in\Gamma}
    \rho_{u,e}(g)\otimes \rho_{v,e}(g).
\]
By Fact~\ref{fact:tensor-average-dual}, this operator is zero unless
\[
    \rho_{v,e}\simeq \rho_{u,e}^{\vee}.
\]
Consequently, if for some edge \(e=uv\) the two representations \(\rho_{u,e}\) and \(\rho_{v,e}\) are not contragredient, then
\[
    Q_e=0,
\]
and hence
\[
    \mathcal{T}(\boldsymbol{\rho})=0.
\]

Thus only those choices of \(\boldsymbol{\rho}\) for which the two representations on every subdivided pair are contragredient can contribute. For such a choice, for each edge \(e=uv\in E\) we have
\[
    d_{\rho_{u,e}}=d_{\rho_{v,e}}
\]
and, by Fact~\ref{fact:dual-class-coefficients},
\[
    \lambda_{\rho_{u,e}}\lambda_{\rho_{v,e}}
    =
    \lambda_{\rho_{u,e}}\lambda_{\rho_{u,e}^{\vee}}
    =
    \lambda_{\rho_{u,e}}\overline{\lambda_{\rho_{u,e}}}
    =
    |\lambda_{\rho_{u,e}}|^2.
\]
Hence
\[
    C(\boldsymbol{\rho})
    =
    \prod_{e=uv\in E}
    d_{\rho_{u,e}}^2
    |\lambda_{\rho_{u,e}}|^2
    \geq 0.
\]
Combining this with
\[
    \mathcal{T}(\boldsymbol{\rho})\geq 0,
\]
we see that every nonzero Fourier contribution to
\[
    t_{\mathrm{Cay}}(H;\Gamma,b)
\]
is non-negative.

Therefore \(t_{\mathrm{Cay}}(H;\Gamma,b)\) is at least the contribution coming from the trivial representation on every incidence. Let \(\rho_{\mathrm{triv}}\) denote the trivial representation of \(\Gamma\). For the choice
\[
    \rho_{v,e}=\rho_{\mathrm{triv}}
    \qquad
    \text{for every }(v,e)\in I(H_0),
\]
we have
\[
    d_{\rho_{\mathrm{triv}}}=1,
    \qquad
    \lambda_{\rho_{\mathrm{triv}}}
    =
    \mathbb{E}_{g\in\Gamma} b(g).
\]
Since \(b=a^{\mathrm{cl}}\), conjugacy averaging preserves the uniform average:
\[
    \mathbb{E}_{g\in\Gamma} b(g)
    =
    \mathbb{E}_{g\in\Gamma} a(g).
\]
Moreover,
\[
    \mathcal{T}(\rho_{\mathrm{triv}},\dots,\rho_{\mathrm{triv}})
    =
    1.
\]
Thus the trivial-representation contribution is exactly
\[
    \left(\mathbb{E}_{g\in\Gamma}a(g)\right)^{|I(H_0)|}
    =
    \left(\mathbb{E}_{g\in\Gamma}a(g)\right)^{2e(H_0)}.
\]
Since all other contributions are non-negative, we conclude that
\[
    t_{\mathrm{Cay}}(H;\Gamma,a^{\mathrm{cl}})
    =
    t_{\mathrm{Cay}}(H;\Gamma,b)
    \geq
    \left(\mathbb{E}_{g\in\Gamma}a(g)\right)^{2e(H_0)}.
\]
Finally,
\[
    t_{\mathrm{Cay}}(K_2;\Gamma,a^{\mathrm{cl}})
    =
    \mathbb{E}_{g\in\Gamma}a^{\mathrm{cl}}(g)
    =
    \mathbb{E}_{g\in\Gamma}a(g).
\]
Since
\[
    e(H)=2e(H_0),
\]
this is precisely
\[
    t_{\mathrm{Cay}}\bigl(H;\Gamma,a^{\mathrm{cl}}\bigr)
    \geq
    t_{\mathrm{Cay}}\bigl(K_2;\Gamma,a^{\mathrm{cl}}\bigr)^{e(H)}.
\]
The theorem follows.
\end{proof}

\section*{Declaration of Generative AI and AI-assisted technologies in the writing process}

During the preparation of this manuscript, the author used ChatGPT for language polishing, editing assistance, and improving readability. The author has reviewed the AI-assisted text and takes full responsibility for the content of this manuscript. All mathematical ideas, statements, proofs, and final verification are the author's own.

\bibliographystyle{abbrv}
\bibliography{ref}

\end{document}